\documentstyle[amscd,amssymb,xypic,verbatim,11pt]{amsart}
\xyoption{all}

\newcommand{\nc}{\newcommand}
\newcommand{\rnc}{\renewcommand}

\nc{\exto}[1]{\stackrel{#1}{\longrightarrow}}
\nc{\dlim}{{\mathop{\lim\limits_{\longrightarrow}}}}
\nc{\lan}{\big\langle}
\nc{\ran}{\big\rangle}

\nc{\kk}{{\mathsf{k}}}
\nc{\ix}{{\mathsf{i}}}
\nc{\jx}{{\mathsf{j}}}

\nc{\C}{{\mathbb{C}}}
\nc{\HH}{{\mathbb{H}}}
\nc{\LL}{{\mathbb{L}}}
\nc{\PP}{{\mathbb{P}}}
\nc{\QQ}{{\mathbb{Q}}}
\nc{\RR}{{\mathbb{R}}}
\nc{\TT}{{\mathbb{T}}}
\nc{\ZZ}{{\mathbb{Z}}}

\nc{\CA}{{\mathcal{A}}}
\nc{\CB}{{\mathcal{B}}}
\nc{\CC}{{\mathcal{C}}}
\nc{\D}{{\mathcal{D}}}
\nc{\CE}{{\mathcal{E}}}
\nc{\CF}{{\mathcal{F}}}
\nc{\CG}{{\mathcal{G}}}
\nc{\CH}{{\mathcal{H}}}
\nc{\CJ}{{\mathcal{J}}}
\nc{\CK}{{\mathcal{K}}}
\nc{\CL}{{\mathcal{L}}}
\nc{\CM}{{\mathcal{M}}}
\nc{\CMF}{{\mathcal{MF}}}
\nc{\CMFI}{{\mathcal{MFI}}}
\nc{\CMFB}{{\mathcal{MFB}}}
\nc{\CN}{{\mathcal{N}}}
\nc{\CO}{{\mathcal{O}}}
\nc{\CQ}{{\mathcal{Q}}}
\nc{\CR}{{\mathcal{R}}}
\nc{\CS}{{\mathcal{S}}}
\nc{\CT}{{\mathcal{T}}}
\nc{\CU}{{\mathcal{U}}}
\nc{\CV}{{\mathcal{V}}}
\nc{\CW}{{\mathcal{W}}}
\nc{\CX}{{\mathcal{X}}}
\nc{\CY}{{\mathcal{Y}}}
\nc{\CZ}{{\mathcal{Z}}}
\nc{\CMo}{{\mathcal{M}^\circ}}
\nc{\Co}{{{C}^\circ}}

\nc{\BY}{{\overline{Y}}}
\nc{\BYD}{{\overline{Y}{}^{|D|}}}
\nc{\OZ}{{\overline{Z}}}
\nc{\bg}{{\bar{g}}}

\nc{\bq}{{\mathbf{q}}}
\nc{\BD}{{\mathbf{D}}}
\nc{\BG}{{\mathbf{G}}}
\nc{\BM}{{\mathbf{M}}}
\nc{\BP}{{\mathbf{P}}}
\nc{\BZ}{{\mathbf{Z}}}
\nc{\BPr}{{\mathsf{P}}}
\nc{\BR}{{\mathbf{R}}}
\nc{\BRO}[1]{{{\mathbf{R}}^{\circ}_{#1}}}
\nc{\BRD}[1]{{{\mathbf{R}}^{|D|}_{#1}}}
\nc{\BRP}[1]{{{\mathbf{R}}^{1}_{#1}}}
\nc{\BRTP}[1]{{{\mathbf{\tilde{R}}}{}^{1}_{#1}}}
\nc{\BS}{{\mathbf{S}}}
\nc{\BMS}{{{\mathbf{M}}^{{s}}}}
\nc{\BMSS}{{{\mathbf{M}}^{{ss}}}}
\nc{\BMZ}{{\mathbf{M}^{\circ}}}
\nc{\BCL}{{\mathbf{L}}}

\nc{\PCC}{{{}^\perp\CC}}

\nc{\ch}{{\mathsf{ch}}}
\nc{\Cl}{{\mathsf{Cliff}}}
\nc{\Clev}{{\mathop{\mathsf{Cliff}}^{\circ}}}

\nc{\FA}{{\mathfrak{A}}}
\nc{\FB}{{\mathfrak{B}}}
\nc{\FF}{{\mathfrak{F}}}
\nc{\FI}{{\mathfrak{I}}}
\nc{\FZ}{{\mathfrak{Z}}}

\nc{\TFA}{{\tilde{\mathfrak{A}}}}
\nc{\TFB}{{\tilde{\mathfrak{B}}}}

\nc{\fa}{{\mathfrak{a}}}
\nc{\fg}{{\mathfrak{g}}}
\nc{\fp}{{\mathfrak{p}}}
\nc{\FD}{{\mathfrak{D}}}
\nc{\FE}{{\mathfrak{E}}}
\nc{\FL}{{\mathfrak{L}}}
\nc{\FM}{{\mathfrak{M}}}
\nc{\FR}{{\mathfrak{R}}}
\nc{\FS}{{\mathsf{S}}}

\nc{\sfc}{{\mathsf{c}}}
\nc{\sfch}{{\mathsf{ch}}}
\nc{\sfh}{{\mathsf{h}}}

\nc{\SK}{{\mathsf{K}}}
\nc{\SO}{{\mathsf{O}}}
\nc{\SQ}{{\mathsf{Q}}}
\nc{\SPV}{{\mathsf{S}^+\mathsf{V}}}
\nc{\SMV}{{\mathsf{S}^-\mathsf{V}}}
\nc{\SPMV}{{\mathsf{S}^\pm\mathsf{V}}}
\nc{\SX}{{S_X}}
\nc{\SY}{{S_Y}}
\nc{\phipsi}{{q}}
\nc{\eps}{\varepsilon}

\nc{\pim}{{\pi_-}}
\nc{\pip}{{\pi_+}}

\nc{\BE}{{{\mathbf E}}}
\nc{\TD}{{\widetilde{\D}}}
\nc{\TFD}{{\widetilde{\FD}}}
\nc{\TE}{{\tilde{\CE}}}
\nc{\TQ}{{\tilde{Q}}}
\nc{\TCA}{{\tilde{\CA}}}
\nc{\TCF}{{\tilde{\CF}}}
\nc{\TCG}{{\tilde{\CG}}}
\nc{\TCH}{{\tilde{\CH}}}
\nc{\TCL}{{\tilde{\CL}}}
\nc{\TF}{{\tilde{F}}}
\nc{\TW}{{\tilde{W}}}
\nc{\TCB}{{\widetilde{\CB}}}
\nc{\TCC}{{\tilde{\CC}}}
\nc{\TCX}{{\tilde{\CX}}}
\nc{\TCY}{{\tilde{\CY}}}
\nc{\TPhi}{{\tilde{\Phi}}}
\nc{\OPhi}{{\bar{\Phi}}}
\nc{\txi}{{\tilde{\xi}}}
\nc{\tp}{{\tilde{p}}}
\nc{\tq}{{\tilde{q}}}
\nc{\tzeta}{{\tilde{\zeta}}}
\nc{\tpi}{{\tilde{\pi}}}

\nc{\HCB}{{\widehat{\CB}}}
\nc{\HCU}{{\widehat{\CU}}}
\nc{\HE}{{\widehat{\CE}}}
\nc{\HS}{{\widehat{S}}}
\nc{\HX}{{\hat{X}}}
\nc{\HY}{{\hat{Y}}}
\nc{\HZ}{{\hat{Z}}}
\nc{\hxi}{{\hat{\xi}}}

\nc{\UH}{{\mathcal{H}}}

\nc{\TM}{{\widetilde{M}}}
\nc{\TCM}{{\widetilde{\CM}}}
\nc{\TS}{{\widetilde{S}}}
\nc{\TU}{{\widetilde{U}}}
\nc{\TX}{{\widetilde{X}}}
\nc{\TY}{{\widetilde{Y}}}
\nc{\TZ}{{\widetilde{Z}}}
\nc{\TYO}{{{\widetilde{Y}}^\circ}}
\nc{\barf}{{\bar{f}}}
\nc{\te}{{\tilde{e}}{}}
\nc{\tf}{{\tilde{f}}}
\nc{\tg}{{\tilde{g}}}
\nc{\ti}{{\tilde{\imath}}}
\nc{\tj}{{\tilde{\jmath}}}
\nc{\ty}{{\tilde{y}}}
\nc{\tphi}{{\tilde{\phi}}}
\nc{\hf}{{\hat{f}}}

\nc{\urho}{{\underline{\rho}}}

\nc{\LRA}{\Leftrightarrow}
\nc{\RA}{\Rightarrow}
\nc{\lotimes}{\mathbin{\mathop{\otimes}\limits^{\mathbb{L}}}}
\nc{\CEnd}{\mathop{\mathcal{E}\mathit{nd}}\nolimits}
\nc{\CExt}{\mathop{\mathcal{E}\mathit{xt}}\nolimits}
\nc{\CHom}{\mathop{\mathcal{H}\mathit{om}}\nolimits}
\nc{\RH}{\mathop{{\mathsf{R}}\Gamma}\nolimits}
\nc{\RGamma}{\mathop{{\mathsf{R}}\Gamma}\nolimits}
\nc{\RHom}{\mathop{\mathsf{RHom}}\nolimits}
\nc{\RCHom}{\mathop{\mathsf{R}\mathcal{H}\mathit{om}}\nolimits}
\nc{\RG}{\mathop{\mathsf{R\Gamma}}\nolimits}
\nc{\Hom}{\mathop{\mathsf{Hom}}\nolimits}
\nc{\Ext}{\mathop{\mathsf{Ext}}\nolimits}
\nc{\End}{\mathop{\mathsf{End}}\nolimits}
\nc{\Tor}{\mathop{\mathsf{Tor}}\nolimits}
\nc{\Tordim}{\mathop{\mathsf{Tor}\text{\rm-}\mathsf{dim}}\nolimits}
\nc{\Hilb}{\mathop{\mathsf{Hilb}}\nolimits}
\nc{\Spec}{\mathop{\mathsf{Spec}}\nolimits}
\nc{\Proj}{\mathop{\mathsf{Proj}}\nolimits}
\nc{\Pic}{\mathop{\mathsf{Pic}}\nolimits}

\nc{\Tw}{\mathop{\mathsf{Tw}}\nolimits}
\nc{\Cone}{\mathop{\mathsf{Cone}}\nolimits}
\nc{\Ker}{\mathop{\mathsf{Ker}}\nolimits}
\nc{\Coker}{\mathop{\mathsf{Coker}}\nolimits}
\nc{\codim}{\mathop{\mathsf{codim}}\nolimits}
\nc{\sing}{{\mathsf{sing}}}
\nc{\supp}{\mathop{\mathsf{supp}}}
\nc{\perf}{{\mathsf{perf}}}
\nc{\rank}{\mathop{\mathsf{rank}}}
\nc{\Pf}{{\mathsf{Pf}}}
\nc{\Gr}{{\mathsf{Gr}}}
\nc{\OGr}{{\mathsf{OGr}}}
\nc{\SGr}{{\mathsf{SGr}}}
\nc{\Flag}{{\mathsf{Fl}}}
\nc{\Kosz}{{\mathsf{Kosz}}}
\nc{\LGr}{{\mathsf{LGr}}}
\nc{\GTGr}{{\mathsf{G_2Gr}}}
\nc{\GT}{{\mathsf{G_2}}}
\nc{\GTF}{{\mathsf{G_2F}}}
\nc{\OF}{{\mathsf{OF}}}
\nc{\Fl}{{\mathsf{Fl}}}
\nc{\Bl}{{\mathsf{Bl}}}
\nc{\GL}{{\mathsf{GL}}}
\nc{\PGL}{{\mathsf{PGL}}}
\nc{\SL}{{\mathsf{SL}}}
\nc{\SP}{{\mathsf{Sp}}}
\nc{\Spin}{{\mathsf{Spin}}}
\nc{\Tot}{{\mathsf{Tot}}}
\nc{\ev}{{\mathsf{ev}}}
\nc{\od}{{\mathsf{odd}}}
\nc{\coev}{{\mathsf{coev}}}
\nc{\id}{{\mathsf{id}}}
\nc{\opp}{{\mathsf{opp}}}
\nc{\PS}{{{\PP^3}}}
\nc{\Qu}{{\mathsf{Q}}}
\nc{\tdim}{\mathop{\Tor\dim}}
\nc{\ecart}{{\fbox{$\scriptstyle\mathsf{EC}$}}}
\nc{\ad}{{\mathop{\mathsf ad}}}
\nc{\gr}{{\mathop{\mathsf gr}}}
\nc{\qgr}{{\mathop{\mathsf qgr}}}
\nc{\tor}{{\mathop{\mathsf tor}}}
\rnc{\mod}{{\mathop{\mathsf mod}}}
\nc{\Mod}{{\mathop{\mathsf Mod}}}
\nc{\Coh}{{\mathop{\mathsf Coh}}}
\nc{\Ab}{{\mathop{\mathcal{A}\mathit{b}}}}
\nc{\QCoh}{{\mathop{\mathsf QCoh}}}

\nc{\AAV}{{\mathcal{AAV}}}

\nc{\Rep}{{\mathsf{Rep}}}

\nc{\Cubics}{{{\mathcal{S}}_3}}
\nc{\VFT}{{{\mathcal{S}}_{14}}}
\nc{\VFTE}{{{\mathcal{N}}_{\mathrm{reg,sm}}}}
\nc{\MX}{{\CM_X}}
\nc{\MY}{{\CM_Y}}
\nc{\MYE}{{\CM_{Y,\CE}}}
\nc{\Yd}{{Y_d}}
\nc{\Yfive}{{Y_5}}
\nc{\Xg}{{X_{2g-2}}}
\nc{\Xtt}{{X_{22}}}
\nc{\Xst}{{X_{16}}}
\nc{\Xtw}{{X_{12}}}
\nc{\Xe}{{X_{8}}}
\nc{\Xf}{{X_{4}}}

\nc{\git}{{/\!\!/\!{}_\chi}}

\theoremstyle{plain}

\newtheorem{theorem}{Theorem}[section]
\newtheorem{conjecture}[theorem]{Conjecture}
\newtheorem{lemma}[theorem]{Lemma}
\newtheorem{proposition}[theorem]{Proposition}
\newtheorem{corollary}[theorem]{Corollary}

\theoremstyle{definition}

\newtheorem{definition}[theorem]{Definition}

\theoremstyle{remark}

\newtheorem{remark}[theorem]{Remark}

\newenvironment{proof}{\noindent{\sf Proof:}}{\qed\medskip}

\title{Derived categories of Fano threefolds}
\author{Alexander Kuznetsov}
\address{\sloppy
\parbox{0.9\textwidth}{
Algebra Section, Steklov Mathematical Institute,
8 Gubkin str., Moscow 119991 Russia
\hfill\\[5pt]
The Poncelet Laboratory, Independent University of Moscow
\hfill\\[5pt]
}}
\email{akuznet@@mi.ras.ru}
\date{}
\thanks{I was partially supported by
RFFI grants 05-01-01034, 07-01-00051 and 07-01-92211,
INTAS 05-1000008-8118,
the Russian Science Support Foundation,
and gratefully acknowledge of the support of the Pierre Deligne fund based on his 2004
Balzan prize in mathematics.}

\begin{document}

\begin{abstract}
We consider the structure of the derived categories of coherent sheaves
on Fano threefolds with Picard number $1$ and describe a strange relation
between derived categories of different threefolds. In the Appendix we discuss
how the ring of algebraic cycles of a smooth projective variety is related
to the Grothendieck group of its derived category.
\end{abstract}

\maketitle

\section{Introduction}

A smooth proper connected algebraic variety $V$ is a Fano variety if
the anticanonical class $-K_V$ on~$V$ is ample. In dimension $1$
the only Fano variety is the projective line $\PP^1$. In dimension $2$
the Fano varieties are known under the name of del Pezzo surfaces.
There are 10 deformation classes of these --- the projective plane
with up to 8 blown up points (in generic position) and the quadric.

An ambitious program of classification of Fano threefolds was initiated
by G.~Fano in the beginning of 20-th century and was mostly accomplished
by V.~Iskovskikh in~1979~\cite{Is}. The final stroke of brush was added by Mukai
and Umemura in 1983~\cite{MU}. In higher dimensions only some pieces of the classification
are known.

There are many interconnections between Fano varieties, which help in the classification problems.
For example, a hyperplane section of a Fano $n$-fold $V$ is a Fano variety of dimension $n-1$ if
the anticanonical class of $V$ is sufficiently large. This is most helpful for classification
of Fano varieties with large anticanonical class. On the other hand, there are many birational
transformations between different Fano varieties of the same dimension. This also is very useful.
For example, the original approach of Fano developed by Iskovskikh was based on these kind of
interconnections (the double projection from a line is one of the most important).

The goal of the present paper is to indicate that there are interconnections between some Fano threefolds
on a higher level, the level of derived categories.
%
%
%
On one side we consider a Fano threefold of index~$2$ and of degree $d$,
an on the other side a Fano threefold of index $1$ and degree $4d + 2$.
Then we find in both derived categories an exceptional pair of vector bundles
and consider the arising semiorthogonal decompositions. The crucial observation
is that the nontrivial components of these decompositions are equivalent.

Actually, this is a very rough formulation. It is well known that in general Fano
varieties have nontrivial moduli spaces, so the components of the derived categories
which we consider in general vary. So, to be more precise one should say that there is
a correspondence in the product of the moduli spaces of both types of Fano threefolds,
points of which correspond to Fano threefolds with equivalent nontrivial components
of derived categories. Investigation of the structure of this correspondence is
an interesting question. We conjecture that this correspondece is dominant over
the both moduli spaces and give some speculations about its structure.

Although it is not so easy to prove that this relation holds (actually, we can check this
only for $d = 3,4,5$, while the cases $d = 1,2$ are still under the question), it is much
more difficult to understand {\em why}\/ such a relation takes place. We believe that
any progress in this direction will be very useful for understanding of the structure
of Fano varieties and for the classification in higher dimensions.

\section{Classification of Fano threefolds}

An excellent modern survey of the classification of Fano varieties is given in~\cite{IP}.
Let us briefly remind those parts of the classification which are the most important for us.
We will work over an algebraically closed field $\kk$ of zero characteristic.



The most important discrete invariant of a Fano variety
is its Picard lattice $\Pic V$ which comes with the intersection form and
with a distinguished element (the canonical class). In this paper we will be mostly
concerned with the (most important) case $\Pic V = \ZZ$. In this case the distinguished
element is represented by a positive integer $i_V$ such that
$$
K_V = - i_V H,
$$
where $H$ is the positive generator of $\Pic V$, and the intersection form
is completely determined by a positive integer
$$
d_V = H^{\dim V}.
$$
These invariants are known as the {\sf index} and the {\sf degree} of $V$ respectively.

The most general result concerning the index is the following

\begin{theorem}[\cite{Fu}]
If $V$ is a Fano variety of index $i_V$ then $i_V \le \dim V + 1$. Moreover
\begin{itemize}
\item if\/ $i_V = \dim V + 1$ then\/ $V = \PP^n$;
\item if\/ $i_V = \dim V$ then\/ $V = Q^n \subset \PP^{n+1}$.
\end{itemize}
\end{theorem}

In particular, for threefolds we have

\begin{corollary}
If $V$ is a Fano threefold then $i_V \le 4$. Moreover
\begin{itemize}
\item if\/ $i_V = 4$ then\/ $V = \PP^3$;
\item if\/ $i_V = 3$ then\/ $V = Q^3 \subset \PP^{4}$.
\end{itemize}
\end{corollary}

Fano threefolds of index 2 are also known as {\sf del Pezzo threefolds}
(since their hyperplane sections are del Pezzo surfaces).

\begin{theorem}
Let $V$ be a Fano threefold with $\Pic V = \ZZ$ of index $i_V = 2$ and of degree $d_V$.
Then
$$
1 \le d_V \le 5
$$
and for each $1 \le d \le 5$ there exists
a unique deformation class of Fano threefolds $Y_d$
with $\Pic Y_d = \ZZ$ of index $2$ and of degree $d$.
They have the following explicit description:
\begin{itemize}
\item $Y_5 = \Gr(2,5) \cap \PP^6 \subset \PP^9$ is a linear section of codimension $3$ of the Grassmannian $\Gr(2,5)$ in the Pl\"ucker embedding;
\item $Y_4 = Q \cap Q' \subset \PP^5$ is an intersection of two $4$-dimensional quadrics;
\item $Y_3 \subset \PP^4$ is a cubic hypersurface;
\item $Y_2 \to \PP^3$ is a double covering ramified in a quartic;
\item $Y_1$ is a hypersurface of degree $6$ in the weighted projective space $\PP(3,2,1,1,1)$.
\end{itemize}
\end{theorem}

Now let $V$ be a Fano threefold of index 1. Its general anticanonical section $S \subset V$
is a K3-surface, which comes with a polarization $H_S = H_{|S}$. It follows that
$$
d_V = H^3 = H_S^2 = 2g_V - 2
$$
for some integer $g_V \ge 2$ which is known as the genus of $V$.

\begin{theorem}
Let $V$ be a Fano threefold with $\Pic V = \ZZ$ of index $i_V = 1$ and of genus $g_V$.
Then
$$
2 \le g_V \le 12,\qquad
g_V \ne 11
$$
and for each $g$ in this range there exists
a unique deformation class of Fano threefolds $X_{2g-2}$
with $\Pic X_{2g-2} = \ZZ$ of index $1$ and of genus $g$.
They have the following explicit description:
\begin{itemize}
\item $X_{22} \subset \PP^{13}$ is the zero locus of a global section of the vector bundle $\Lambda^2\CU^* \oplus \Lambda^2\CU^* \oplus \Lambda^2\CU^*$
on the Grassmannian $\Gr(3,7)$, where $\CU$ denotes the tautological rank $3$ bundle;
\item $X_{18} = \GTGr(2,7) \cap \PP^{11} \subset \PP^{13}$, where $\GTGr(2,7)$ is the minimal compact
homogeneous space for the simple algebraic group of type $\GT$, which can be realized as the zero locus
of a global section of the vector bundle $\CU^\perp(1)$ on the Grassmannian $\Gr(2,7)$;
\item $X_{16} = \PP^{10} \cap \SGr(3,6) \subset \PP^{13}$ is a linear section of codimension $3$
of the symplectic Lagrangian Grassmannian $\SGr(3,6)$ in the Pl\"ucker embedding;
\item $X_{14} = \PP^{9} \cap \Gr(2,6) \subset \PP^{14}$ is a linear section of codimension $5$
of the Grassmannian $\Gr(2,6)$ in the Pl\"ucker embedding;
\item $X_{12} = \PP^{8} \cap \OGr_+(5,10) \subset \PP^{15}$ is a linear section of codimension $7$
of the connected component of the orthogonal Lagrangian Grassmannian $\OGr_+(5,10)$ in the half-spinor embedding;
\item $X_{10} = \PP^7 \cap Q \cap \Gr(2,5) \subset \PP^9$ is a quadric section of a linear section of codimension $2$
of the Grassmannian $\Gr(2,5)$ in the Pl\"ucker embedding; or $X_{10} \to Y_5$ is a twofold covering ramified in a quadric;
\item $X_8 = Q \cap Q' \cap Q'' \subset \PP^6$ is an intersection of three $5$-dimensional quadrics;
\item $X_6 = Q \cap F_3 \subset \PP^5$ is an intersection of a quadric and a cubic;
\item $X_4 \subset \PP^4$ is a quartic; or $X_4 \to Q$ is a double cover of a quadric $Q \subset \PP^4$
ramified in the intersection of $Q$ with a quartic;
\item $X_2 \to \PP^3$ is a double covering ramified in a sextic.
\end{itemize}
\end{theorem}

We will also need the following result of S.~Mukai.

\begin{theorem}[\cite{M}]\label{muth}
Assume that the genus $g$ of a Fano threefold $X_{2g-2}$ can be represented
as a product $g = r\cdot s$ of two integers. Then on $X_{2g-2}$ there exists
a unique stable vector bundle $\CE_r$ of rank $r$ with $c_1(\CE_r) = -H$ and $c_2(\CE_r) = \frac12 H^2 + (r-s)L$,
where $L$ is the class of a line on $X_{2g-2}$. Moreover, $\CE_r$ is exceptional and $H^\bullet(X,\CE_r) = 0$.
\end{theorem}

The definition of exceptional bundles is given below (see Definition~\ref{de}).

\begin{remark}
When applied to the Fano threefolds $X_{2g-2}$ with $g \ge 6$ this Theorem
gives the vector bundles which are the restrictions of the dual tautological
bundles from the corresponding Grassmannians.
\end{remark}

\section{Relation of derived categories}

For an algebraic variety $V$ we denote by $\D^b(V)$ the bounded derived category of coherent sheaves on~$V$.
Recall that $\D^b(V)$ is triangulated. We will always assume that $V$ is smooth and projective.

\begin{definition}[\cite{BK,BO1}]
A semiorthogonal decomposition of a triangulated category $\CT$ is a sequence of
full triangulated subcategories $\CC_1,\dots,\CC_n$ in $\CT$ such that
$\Hom_{\CT}(\CC_i,\CC_j) = 0$ for $i > j$
and for every object $T \in \CT$ there exists a chain of morphisms
$0 = T_n \to T_{n-1} \to \dots \to T_1 \to T_0 = T$ such that
the cone of the morphism $T_k \to T_{k-1}$ is contained in $\CC_k$
for each $k=1,2,\dots,n$.
\end{definition}

We will write $\CT = \langle \CC_1,\CC_2,\dots,\CC_n \rangle$ for a semiorthogonal
decomposition of a triangulated category $\CT$ with components $\CC_1,\CC_2,\dots,\CC_n$.

\begin{definition}[\cite{B}]\label{de}
An object $F \in \CT$ is called {\em exceptional}\/ if $\Hom(F,F)=\kk$
and $\Ext^p(F,F)=0$ for all $p\ne 0$. A collection of exceptional
objects $(F_1,\dots,F_m)$ is called {\em exceptional}\/ if
$\Ext^p(F_l,F_k)=0$ for all $l > k$ and all $p\in\ZZ$.
\end{definition}

\begin{proposition}[\cite{BO1}]
Any exceptional collection $F_1,\dots,F_m$ in $\D^b(V)$
gives a semiorthogonal decomposition
$$
\D^b(V) = \langle \CC , F_1, \dots, F_m \rangle
$$
where $\CC = \langle F_1, \dots, F_m \rangle^\perp = \{F \in \D^b(V)\ |\ \text{$\Ext^\bullet(F_k,F) = 0$ for all $1 \le k \le m$}\}$
and all the other components are the subcategories of $\D^b(V)$ generated by $F_k$ {\rm(}each of these is equivalent to
$\D^b(\kk)$, the derived category of $\kk$-vector spaces{\rm)}.
\end{proposition}

Now let $V$ be a Fano variety of index $i = i_V$.
The following result is well known.

\begin{lemma}
Let $V$ be a Fano variety of index $i = i_V$. Then the collection $\CO_V,\CO_V(H),\dots,\CO_V((i-1)H)$
in $\D^b(V)$ is exceptional.
\end{lemma}
\begin{proof}
We note that $H^p(V,\CO_V(-kH)) = 0$ for $1 \le k \le i-1$ and all $p$ by the Kodaira vanishing theorem,
hence $\Ext^\bullet(\CO_V(lH),\CO_V(kH)) = 0$ for $0 \le k < l \le i-1$. Similarly,
$H^p(V,\CO_V) = 0$ for all $p > 0$ by the Kodaira vanishing theorem, while $H^0(V,\CO_V) = \kk$
since $V$ is connected. Therefore all line bundles on $V$ are exceptional.
\end{proof}

\begin{corollary}
For any Fano variety $V$ we have the following semiorthogonal decomposition
$$
\D^b(V) = \langle \CB_V,\CO_V,\CO_V(H),\dots,\CO_V((i-1)H) \rangle,
$$
where $i = i_V$ is the index of $V$ and $\CB_V = \{F \in \D^b(V)\ |\ \text{$H^\bullet(V,F(-kH)) = 0$ for all $0 \le k \le i-1$}\}$.
\end{corollary}

In particular, for Fano threefolds $Y_d$ of index $2$ we obtain a semiorthogonal decomposition
$$
\D^b(Y_d) = \langle \CB_{Y_d},\CO_V,\CO_V(H) \rangle.
$$

For Fano threefolds $X_{2g-2}$ of index $1$ and even genus $g = 2t$ we consider the vector bundle $\CE_2$ of rank~$2$
provided by Theorem~\ref{muth}. Applying this Theorem we deduce the following

\begin{lemma}
Let $X = X_{2g-2}$ be a Fano threefolds of index $1$ and even genus $g = 2t$.
Let $\CE = \CE_2$ be the vector bundle of rank $2$ on $X$ constructed in Theorem~$\ref{muth}$.
Then $(\CE,\CO_X)$ is an exceptional pair on $X$ and we have a semiorthogonal decomposition
$$
\D^b(X_{2g-2}) = \langle \CA_{X_{2g-2}}, \CE , \CO_{X_{2g-2}}\rangle,
$$
where $\CA_{X_{2g-2}} = \{F \in \D^b(X_{2g-2})\ |\ H^\bullet(X_{2g-2},F) = \Ext^\bullet(\CE,F) = 0\}$.
\end{lemma}

Now we can formulate the Conjecture.

\begin{conjecture}\label{mainc}
Let $\CMF^i_d$ be the moduli spaces of Fano threefolds of index $i$ and degree $d$.
Then there is a correspondence $Z_d \subset \CMF^2_d \times \CMF^1_{4d+2}$ which is dominant
over each factor and such that for any point $(Y_d,X_{4d+2}) \in Z_d$ there is an equivalence of categories
$$
\CA_{X_{4d+2}} \cong \CB_{Y_d}.
$$
\end{conjecture}

The main support for this conjecture is provided by the following results.

\begin{theorem}\label{maint}
Conjecture~$\ref{mainc}$ is true for $d = 3,4,5$.
\end{theorem}

A proof of Theorem~\ref{maint} will be given in the next section.
It consists of a case-by-case analysis, see Corollaries~\ref{cor5}, \ref{cor4} and~\ref{cor3}.
And now we are going to show that the numerical Grothendieck groups of categories
$\CA_{X_{4d+2}}$ and $\CB_{Y_d}$ are isomorphic.

For a triangulated category $\CT$ we denote by $K_0(\CT)$ its Grothendieck group.
It comes with the bilinear Euler form
$$
\chi([F],[G]) = \sum_p (-1)^p \dim\Ext^p(F,G).
$$
The {\sf numerical Grothendieck group}\/ $K_0(\CT)_{num}$ is defined as the quotient
$K_0(\CT)_{num} := K_0(\CT)/\Ker\chi$.
If $\CT = \D^b(V)$ we write $K_0(V)$ and $K_0(V)_{num}$ instead of $K_0(\D^b(V))$ and $K_0(\D^b(V))_{num}$ for brevity.

\begin{proposition}\label{rr}
For all $1 \le d \le 5$ there is an isomorphism of numerical Grothendieck groups
$$
K_0(\CA_{X_{4d+2}})_{num} \cong K_0(\CB_{Y_d})_{num}
$$
compatible with the Euler bilinear forms.
\end{proposition}
\begin{proof}
A computation based on the Riemann--Roch Theorem.

Note that by~\ref{k0f3} the Chern character map $\ch:K_0(X_{2g-2})_{num} \to \oplus_{p=0}^3 H^{2p}(X_{2g-2},\QQ)$
identifies the numerical Grothendieck group $K_0(X_{2g-2})_{num}$ with the lattice generated by elements
$$
\ch(\CO_{X_{2g-2}}) = 1,\quad
\ch(\CO_H) = H - (g-1)L + \frac{g-1}3 P,\quad
\ch(\CO_L) = L + \frac12 P,\quad
\ch(\CO_P) = P,
$$
where $P$ is the class of a point, and by Riemann--Roch the Euler form can be expressed as
$$
\chi(u,v) = \chi_0(u^*\cap v),
$$
where $u \mapsto u^*$ is the involution of $\oplus_{p=0}^3 H^{2p}(X_{2g-2},\QQ)$ given by $(-1)^p$-multiplication
on $H^{2p}(X_{2g-2},\QQ)$, and $\chi_0$ is given by the formula
$$
\chi_0(x + yH + zL + wP) = x + \frac{g+11}6 y + \frac12 z + w.
$$
On the other hand, we have $\ch(\CO_{X_{2g-2}}) = 1$, and using Theorem~\ref{muth} it is easy to compute
$$
\ch(\CE) = 2 - H + \frac{g-4}2 L - \frac{g-10}{12} P.
$$
It follows that
$$
K_0(\CA_{X_{2g-2}})_{num} =
\left\langle 1, 2 - H + \frac{g-4}2 L - \frac{g-10}{12} P \right\rangle^\perp\! =
\left\langle 1 - \frac g2 L + \frac{g-4}4P, H - \frac{3g - 6}2 L + \frac{7g - 40}{12} P \right\rangle.
$$
Computing the form $\chi$ on the base vectors we conclude that
$$
K_0(\CA_{X_{2g-2}})_{num} \cong \ZZ^2,
\qquad
\chi_\CA = \left(\begin{matrix} 1 - g/2 & - g/2 \\ 3-g & 1-g \end{matrix}\right).
$$
 as a lattice with a bilinear form.

Similarly, by~\ref{k0f3} the Chern character map $\ch:K_0(Y_d)_{num} \to \oplus_{p=0}^3 H^{2p}(Y_d,\QQ)$
identifies $K_0(Y_d)_{num}$ with the lattice generated by elements
$$
\ch(\CO_{Y_d}) = 1,\quad
\ch(\CO_H) = H - \frac d2 L + \frac d6 P,\quad
\ch(\CO_L) = L,\quad
\ch(\CO_P) = P,
$$
and it is easy to see that $\chi_0$ is given by the formula
$$
\chi_0(x + yH + zL + wP) = x + \frac{d+3}3 y + z + w.
$$
On the other hand, we have $\ch(\CO_{Y_d}) = 1$, $\ch(\CO_{Y_d}(H)) = 1 + H + \frac d2 L + \frac d6 P$, so
it follows that
$$
K_0(\CB_{Y_d})_{num} =
\left\langle 1, 1 + H + \frac d2 L + \frac d6 P \right\rangle^\perp =
\left\langle 1 - L, H - \frac{d}2 L + \frac{d-6}{6} P \right\rangle.
$$
Computing the form $\chi$ on the base vectors we conclude that
$$
K_0(\CB_{Y_d})_{num} \cong \ZZ^2,
\qquad
\chi_\CB = \left(\begin{matrix} -1 & - 1 \\ 1-d & -d \end{matrix}\right).
$$
 as a lattice with the bilinear form.

A direct check shows that for the map $\ZZ^2 \to \ZZ^2$ given by the matrix
$$
A = \left(\begin{matrix} 0 & 1 \\ -1 & -2 \end{matrix}\right)
$$
we have
$$
A^T\cdot\chi_\CB\cdot A = \left(\begin{matrix} -d & -1-d \\ 1-2d & -1-2d \end{matrix}\right)
$$
which for $g = 2d+2$ coincides with the matrix of $\chi_\CA$. Thus the map $A$
gives a required isomorphism $K_0(\CB_{Y_d})_{num} \cong K_0(\CA_{X_{4d+2}})_{num}$
compatible with the Euler forms.
\end{proof}

\section{Cases $d \ge 3$}

In this section we prove Theorem~\ref{maint} by a case-by-case analysis.

\subsection{The case $d=5$}

Recall that the del Pezzo threefold $Y_5$ of degree $5$ is rigid, so that the moduli space $\CMF^2_5$ is a point.
On the contrary, Fano threefolds $X_{22}$ of genus $12$ have a 6-dimensional moduli space $\CMF^1_{22}$.
We will show that the correspondence $Z_5 \subset \CMF^2_5 \times\CMF^1_{22}$ is the whole product.
In the other words, we are going to show that for any $X_{22}$ and for the unique $Y_5$
there is an equivalence $\CA_{X_{22}} \cong \CB_{Y_5}$. For this we give an explicit
description of both categories in question.

For the $X_{22}$ the description is based on the following result.
Let $\CE_2$ and $\CE_3$ be the vector bundles of rank $2$ and $3$ on $X = X_{22}$
provided by Theorem~\ref{muth} for the factorizations $g_X = 12 = 2\cdot 6 = 3\cdot 4$.
Let $W = H^0(X,\CE_3^*)^* \cong \kk^7$, so that $X \subset \Gr(3,W)$.

\begin{theorem}[\cite{K1,K2}]
The bundles $W/\CE_3\otimes\CO_X(-1),\CE_3^*\otimes\CO_X(-1),\CE_2,\CO_X$ form a full exceptional collection, so that
$$
\D^b(X_{22}) = \langle W/\CE_3\otimes\CO_X(-1),\CE_3^*\otimes\CO_X(-1),\CE_2,\CO_X \rangle.
$$
Moreover, $\Hom(W/\CE_3\otimes\CO_X(-1),\CE_3^*\otimes\CO_X(-1)) = \kk^3$, $\Ext^{\ne 0}(W/\CE_3\otimes\CO_X(-1),\CE_3^*\otimes\CO_X(-1)) = 0$, so that
$$
\CA_{X_{22}} \cong \D^b(\Qu_3),
$$
where $\Qu_3 = \left(\xymatrix@1{ \bullet \ar[r] \ar@<1ex>[r] \ar@<-1ex>[r] & \bullet }\right)$ is the Kronecker quiver with $3$ arrows.
\end{theorem}

On the other hand, for the $Y_5$ a description of the derived category was given by Orlov.
Let $\CU$ be the restriction to $Y = Y_5$ of the tautological bundle from the Grassmannian $\Gr(2,5)$ to $Y$.
Let $W' = H^0(Y,\CU^*)^* \cong \kk^5$, so that $Y \subset \Gr(2,W')$.

\begin{theorem}[\cite{Or}]
The bundles $W'/\CU\otimes\CO_Y(-1),\CU,\CO_Y,\CO_Y(1)$ form a full exceptional collection, so that
$$
\D^b(Y_5) = \langle W'/\CU\otimes\CO_Y(-1),\CU,\CO_Y,\CO_Y(1) \rangle.
$$
Moreover, $\Hom(W'/\CU\otimes\CO_Y(-1),\CU) = \kk^3$, $\Ext^{\ne 0}(W'/\CU\otimes\CO_Y(-1),\CU) = 0$, so that
$\CB_{Y_5} \cong \D^b(\Qu_3)$.
\end{theorem}

From these two results we immediately deduce the required equivalence.

\begin{corollary}\label{cor5}
For any Fano threefold $X_{22}$ of genus $12$ and for the unique del Pezzo threefold $Y_5$ of degree $5$
there is an equivalence of categories $\CA_{X_{22}} \cong \D^b(\SQ_3) \cong \CB_{Y_5}$.
\end{corollary}

\subsection{The case $d=4$}

Recall that the moduli space $\CMF^2_4$ of del Pezzo threefolds $Y_4$ of degree $4$ is isomorphic to the moduli space $\CM_2$
of curves of genus $2$. The isomorphism is constructed as follows. Let $Y = Y_4 = Q \cap Q' \subset \PP^5$.
Consider the pencil of quadrics $\{Q_\lambda\}_{\lambda\in\PP^1}$ generated by $Q$ and $Q'$. If $Y$ is smooth then the generic $Q_\lambda$ is smooth
and there are precisely $6$ distinct points $\lambda_1,\dots,\lambda_6 \in \PP^1$ for which the quadric $Q_\lambda$ is degenerate.
Consider the twofold covering $C(Y) \to \PP^1$ ramified at the points $\lambda_i$. Then $C(Y)$ is a smooth curve of genus $2$.

\begin{theorem}[\cite{BO1,K6}]
The map $\CMF^2_4 \to \CM_2$, $Y \mapsto C(Y)$ is an isomorphism.
Moreover, there is an equivalence $\CB_{Y_4} \cong \D^b(C(Y_4))$.
\end{theorem}

Our goal is to show that $X_{18}$ threefolds behave in a similar fashion.
Indeed, recall that by definition any $X = X_{18}$ is a linear section
of codimension $2$ in $\GTGr(2,7)$. Let $\{\CX_\lambda\}_{\lambda\in\PP^1}$ be the pencil of hyperplane sections
of $\GTGr(2,7)$ passing through $X$. Since the projective dual of $\GTGr(2,7)$ is a hypersurface of degree $6$,
it follows that there are precisely $6$ distinct points $\lambda_1,\dots,\lambda_6 \in \PP^1$ for which
$\CX_\lambda$ is singular. Consider the twofold covering $C(X) \to \PP^1$ ramified at the points $\lambda_i$.
Then $C(X)$ is a smooth curve of genus $2$.
Thus we obtain a map $\CMF^1_{18} \to \CM_2$, $X \mapsto C(X)$.

\begin{theorem}[\cite{K5}]
There is an equivalence $\CB_{X_{18}} \cong \D^b(C(X_{18}))$.
\end{theorem}

Combining these results we obtain the case $d = 4$.
Let $Z_4 \subset \CMF^2_4 \times\CMF^1_{18}$ be the graph of the morphism
$\CMF^1_{18} \to \CM_2 \cong \CMF^2_4$.

\begin{corollary}\label{cor4}
For any pair of threefolds $(Y_4,X_{18}) \in Z_4 \subset \CMF^2_4 \times\CMF^1_{18}$ we have an equivalence of categories $\CA_{X_{18}} \cong \CB_{Y_4}$.
\end{corollary}

\subsection{The case $d=3$}

While in the previous two cases we were able to describe the categories under the question explicitly,
for $d = 3$ this is no longer possible. We can only prove an equivalence in this case.

Recall that by definition any $X = X_{14}$ is a linear section of codimension $5$ in $\Gr(2,6)$.
Let $W$ be the six-dimensional vector space, so that $\Gr(2,6) = \Gr(2,W) \subset \PP(\Lambda^2W)$.
Then $X$ can be described by a $5$-dimensional subspace $A \subset \Lambda^2W^*$ or, equivalently,
by an injective map $\alpha:A \to \Lambda^2W$, where $A$ is a fixed vector space of dimension $5$.
Let us denote the corresponding threefold $X_{14}$ by $X(\alpha)$.

On the other hand, consider the whole space $\PP(\Lambda^2W^*)$, the space of skew-symmetric forms on $W$.
Consider the hypersurface therein consisting of degenerate skew-forms. It is well known that the equation
of this hypersurface is given by the Pfaffian polynomial. It follows that
it is a cubic hypersurface, which is denoted by $\Pf(W)$ and is called the Pfaffian variety.
Certainly, the Pfaffian variety is singular, its singular locus coincides with the set of all skew-forms
of rank $2$ on $W$, that is with the Grassmannian $\Gr(2,W^*) \subset \Pf(W) \subset \PP(\Lambda^2W^*)$.
However, the codimension of the singular locus in $\PP(\Lambda^2W^*)$ is $14-8 = 6$, so for generic $\alpha$
the preimage of $\Pf(W)$ in $\PP(A)$ is a smooth cubic hypersurface, which we denote $Y(\alpha)$.
Further, associating with a degenerate skew-form on $W$ its kernel, defines a rank 2 subbundle $K \subset W\otimes\CO$
on the smooth locus of $\Pf(W)$. Let $E(\alpha) = \alpha^*K\otimes\CO_{Y(\alpha)}(1)$, where in the right-hand-side
we consider $\alpha$ as a map $Y(\alpha) \to \Pf(W)$ and $\CO_{Y(\alpha)}(1)$ is the restriction of $\CO_{\PP(A)}(1)$.

\begin{theorem}[\cite{K3,K9}]\label{th3}
The map $X(\alpha) \mapsto (Y(\alpha),E(\alpha))$ gives an isomorphism
of the moduli space $\CMF^1_{14}$ of Fano threefolds $X_{14}$ of genus $8$
and the moduli space $\CMFI^2_3(2)$ of pairs $(Y,E)$, where $Y$ is a smooth cubic threefold
and $E$ is a stable vector bundle on $Y$ of rank $2$, $c_1(E) = 0$, $c_2(E) = 2L$ with $H^1(Y,E(-1)) = 0$.
For every $\alpha$ there is an equivalence of categories
$\CB_{X_{14}(\alpha)} \cong \CA_{Y_3(\alpha)}$.
\end{theorem}

\begin{remark}
The bundles $E$ in the statement of the Theorem are known as {\sf instanton bundles of charge~$2$}.
\end{remark}

Let $Z_3 \subset \CMF^2_3 \times\CMF^1_{14}$ be the graph of the morphism
$\CMF^1_{14} \cong \CMFI^2_3(2) \to \CMF^2_3$.

\begin{corollary}\label{cor3}
For any pair of threefolds $(Y_3,X_{14}) \in Z_3 \subset \CMF^2_3 \times\CMF^1_{14}$ we have an equivalence of categories $\CA_{X_{14}} \cong \CB_{Y_3}$.
\end{corollary}

\subsection{Geometrical correspondences}

Actually the proof of Theorem~\ref{th3} in~\cite{K3} gives more than just an equivalence of categories.
It gives also a geometrical correspondence between $X(\alpha)$ and $Y(\alpha)$.

Let $\PP_{X(\alpha)}(\CE)$ be the projectivization of the exceptional rank $2$ bundle on $X(\alpha)$.
Since $\CE$ is the restriction of the tautological bundle from the Grassmannian $\Gr(2,W)$, we have
a canonical map $\PP_{X(\alpha)}(\CE) \to \PP(W)$. On the other hand, one can check that
we have an isomorphism $H^0(Y(\alpha),E(\alpha)^*\otimes\CO_{Y(\alpha)}(1)) \cong W^*$,
hence we have also a canonical map $\PP_{Y(\alpha)}(E(\alpha)) \to \PP(W)$.

\begin{theorem}[\cite{K3}]\label{d3geom}
The images of $\PP_{X(\alpha)}(\CE)$ and $\PP_{Y(\alpha)}(E(\alpha))$ in $\PP(W)$ coincide
with a quartic hypersurface $M \subset \PP(W)$ singular along a curve $C \subset M$ of genus $26$.
The maps $\PP_{X(\alpha)}(\CE) \to M$ and $\PP_{Y(\alpha)}(E(\alpha)) \to M$ are small contractions
and induce isomorphisms over the complement of $C$. Moreover, the induced birational isomorphism
$\xymatrix@1{\PP_{X(\alpha)}(\CE)\  \ar@{-->}[r] & \ \PP_{Y(\alpha)}(E(\alpha))}$ is a flop.
\end{theorem}

\begin{remark}
The hypersurface $M \subset \PP(W)$ is known as the {\em da Palatini quartic}.
The curve $C$ parameterizes lines on $X(\alpha)$ and at the same time jumping lines
for $E(\alpha)$ on $Y(\alpha)$ (that is lines $L \subset Y(\alpha)$ for which
$E(\alpha)_{|L} \cong \CO_L(1) \oplus \CO_L(-1)$).
\end{remark}

We expect that some generalization of this result should hold for other values of $d$.
For example, let $\CMFB^1_{4d+2}(t)$ be the moduli space of pairs $(X,F)$,
where $X$ is a Fano threefold of index $1$ and degree $4d + 2$, and $F$ is
a stable vector bundle on $X$ of rank $2$ with $c_1(F) = -H$, $c_2(F) = (d + 2 + t)L$
(note that for $t = 0$ by Theorem~\ref{muth} there is only one such bundle,
the exceptional bundle $\CE_2$, hence $\CMFB^1_{4d+2}(0) = \CMF^1_{4d+2}$).
Using the Riemann--Roch (see the proof of Proposition~\ref{rr})
one can check that for any $d$, any $t$
and any $(X,F) \in \CMF^1_{4d+2}(t)$ we have
$$
\dim H^0(X,F^*) = d + 3 - t,
$$
so that we have a map $\PP_{X}(F) \to \PP^{d+2-t}$. Moreover,
since $c_1(F^*) = H$, $c_2(F^*) = (d+2-t)L$,
it follows that the degree of the image of $\PP_{X}(F)$ in $\PP^{d+2-t}$ is
$$
\deg\PP_{X}(F) = c_1(F^*)^3 - 2c_1(F^*)c_2(F^*) = H^3 - 2(d+2-t)HL = (4d + 2) - 2(d+2-t) = 2d - 2 + 2t.
$$
Similarly, let
$\CMFI^1_{d}(k)$ be the moduli space of pairs $(Y,E)$,
where $Y$ is a Fano threefold of index $2$ and degree $d$, and $E$ is
an instanton bundle of charge $k$ on $Y$, that is
a stable vector bundle of rank $2$ with $c_1(E) = 0$, $c_2(E) = kL$ and $H^1(Y,E(-1)) = 0$
(see~\cite{K3}). Using the Riemann--Roch (see the proof of Proposition~\ref{rr})
one can check that for any $d$, any $k$ and any $(Y,E) \in \CMFI^2_d(k)$ we have
$$
\dim H^0(Y,E^*(1)) = 2d - 2k + 4,
$$
so that we have a map $\PP_{Y}(E) \to \PP^{2d-2k+3}$. Moreover,
since $c_1(E^*(1)) = 2H$, $c_2(E(1)) = H^2 + kL$,
it follows that the degree of the image of $\PP_{Y}(E)$ in $\PP^{2d-2k+3}$ is
$$
\deg\PP_{Y}(E) = c_1(E^*(1))^3 - 2c_1(E^*(1))c_2(E^*(1)) = 8H^3 - 4H(H^2 + kL) = 8d - 4(d+k) = 4d - 4k.
$$
Note that whenever $d + 1 = 2k - t$ the dimensions and the degrees coincide.


\begin{conjecture}
For each $1 \le d \le 5$ there are integers $k,t \ge 0$ satisfying $d + 1 = 2k - t$ for which
there is an isomorphism $\xi:\CMFI^2_d(k) \cong \CMFB^1_{4d+2}(t)$
such that for $(X,F) = \xi(Y,E)$ there is an isomorphism $h:H^0(Y,E^*(1)) \stackrel\sim\to H^0(X,F^*)$
and a birational isomorphism
$\theta:\xymatrix@1{\PP_{X}(F)\  \ar@{-->}[r] & \ \PP_{Y}(E)}$
such that the diagram
$$
\xymatrix{
\PP_{X}(F)\  \ar@{-->}[rr]^-\theta \ar[d] && \ \PP_{Y}(E) \ar[d] \\
\PP(H^0(X,F^*)^*) \ar[rr]^-h_-\cong && \PP(H^0(Y,E^*(1))^*)
}
$$
commutes. Moreover, there is an equivalence $\CA_X \cong \CB_Y$, that is $Z_d \subset \CMF^2_d\times\CMF^1_{4d+2}$
is the image of the graph of the isomorphism $\xi:\CMF^2_d(k) \to \CMFB^1_{4d+2}(t)$.
\end{conjecture}

\begin{remark}
For $d = 3$ by Theorem~\ref{d3geom} we should take $k = 2$, $t = 0$.
For $d = 5$ we expect $k =4$, $t = 2$ will work.
\end{remark}



\section{Appendix. The Grothedieck group and algebraic cycles}

Let $X$ be a smooth projective variety of dimension $n$.
Let $A^p(X)$ denote the group of algebraic cycles on $X$ of codimension $p$ modulo rational equivalence.
Let $A^\bullet(X) = \oplus_{p=0}^n A^p(X)$ be the Chow ring. Let $K_0(X)$ be the Grothendieck group of the category
of coherent sheaves on $X$ (equivalently, of the derived category $\D^b(X)$).
Consider the Chern character map $\ch:K_0(X) \to A^\bullet(X)\otimes\QQ$.
It is well known that $\ch$ induces an isomorphism of $\QQ$-vector
spaces $K_0(X)\otimes\QQ \to A^\bullet(X)\otimes\QQ$.

On the other hand, consider on both sides the numerical equivalence.
Recall that an algebraic cycle $a \in A^p(X)$ is numerically equivalent to zero,
if it lies in the kernel of the bilinear intersection form:
$$
\xymatrix@1{A^\bullet(X)\otimes A^\bullet(X) \ar[r]^-\cdot & A^\bullet(X) \ar[r]^-{\mathrm{pr}} & A^n(X) \ar[r]^-\deg & \ZZ}.
$$
In other words, if its intersection with any cycle in $A^{n-p}(X)$ is zero.
Let $A^\bullet(X)_{num} = \oplus A^p(X)_{num}$ be the ring of algebraic cycles
modulo the numerical equivalence. Note that any torsion class in $A^\bullet(X)$ is numerically
trivial, hence $A^\bullet(X)_{num}$ is torsion free.

Similarly, a class $v \in K_0(X)$ is numerically equivalent to zero,
if it lies in the kernel of the Euler bilinear form:
$$
\chi:\xymatrix@1{K_0(X) \otimes K_0(X) \ar[r] & \ZZ},\qquad
\chi([F],[G]) = \sum_i (-1)^i\dim\Ext^i(F,G).
$$
Let $K_0(X)_{num} := K_0(X)/\Ker\chi$ be the numerical Grothendieck group.

The Riemann-Roch formula shows that the kernel of the Euler form
coinsides with the preimage under the Chern character map of the subring of $A^\bullet(X)\otimes\QQ$
consisting of numerically trivial algebraic cycles. It follows that $\ch$ descends to a map
$K_0(X)_{num} \to A^\bullet(X)_{num} \otimes\QQ$ which we denote by $\ch$ as well,
and induces an isomorphism of $\QQ$-vector spaces $K_0(X)_{num}\otimes\QQ \cong A^\bullet(X)_{num}\otimes\QQ$.

For any $p$-cycle $Z = \sum a_iS_i$ we define $[\CO_Z] := \sum a_i[\CO_{S_i}] \in K_0(X)$.
Note that $[\CO_Z]$ really depends on the cycle $Z$, not only on its rational or numerical equivalence class.
A little bit later we will show how one can get rid of this dependance (see Remark~\ref{grpoz}).

\begin{definition}\label{AK}
We will say that a smooth projective $n$-dimensional variety $X$ is {\sf AK-compatible}\/
if for any collection of cycles $Z_p^i$ on $X$, $0 \le p \le n$, $1 \le i \le m_p$
such that $\codim Z_p^i = p$ and $\{Z^p_1,Z^p_2,\dots,Z^p_{m_p}\}$ is a basis in $A^p(X)_{num}$
the classes $[\CO_{Z^p_i}]$, $0 \le p\le n$, $1 \le i \le m_p$ form a $\ZZ$-basis in $K_0(X)_{num}$.
\end{definition}

If an algebraic variety $X$ is AK-compatible, then one can easily describe its numerical
Grothendieck group by choosing some bases in the groups of algebraic cycles and considering
their structure sheaves. Certainly such a description may be useful in many cases.
The goal of this section is to find some easily verifiable criterion for AK-compatibility.


We start with some preparations.
Consider the following two filtrations on $K_0(X)_{num}$.
The first one is induced by the codimension filtration on $A^\bullet(X)$:
$$
F^pK_0(X)_{num} = \ch^{-1}(\oplus_{q=p}^n A^q(X)_{num}\otimes\QQ).
$$
The second one is induced by the codimension of support:
$$
S^pK_0(X)_{num} = \langle\ [G]\ |\ \codim\supp(G) \ge p\ \rangle,
$$
where $\langle\quad\rangle$ stands for the linear span.
We will call the filtration $F^\bullet$ the {\sf induced filtration},
and $S^\bullet$ the {\sf support filtration}.
Let $\gr_F^p K_0(X)_{num}$ and $\gr_S^p K_0(X)_{num}$ be the graded factors of these filtrations.

Note that for any $G \in \D^b(X)$ with $\codim\supp(G) \ge p$ we have $\ch(G) \in \oplus_{q \ge p} A^q(X)_{num}\otimes\QQ$,
hence $S^pK_0(X)_{num} \subset F^pK_0(X)_{num}$ and we have the following commutative diagram
$$
\xymatrix{
S^pK_0(X)_{num} \ar@{^{(}->}[r] &
F^pK_0(X)_{num} \ar[r]^-\ch &
\oplus_{q \ge p} A^q(X)_{num}\otimes\QQ \\
S^{p+1}K_0(X)_{num} \ar@{^{(}->}[r] \ar@{^{(}->}[u] &
F^{p+1}K_0(X)_{num} \ar[r]^-\ch \ar@{^{(}->}[u] &
\oplus_{q \ge p + 1} A^q(X)_{num}\otimes\QQ  \ar@{^{(}->}[u]
}
$$
Passing to the graded factors we obtain a chain of maps
$$
\xymatrix@1{
\gr^p_SK_0(X)_{num}\ \ar[rr]^{i_p} && \ \gr^p_FK_0(X)_{num}\ \ar[rr]^-{\ch_p} && \ A^p(X)_{num}\otimes\QQ
}
$$
Here $\ch_p$ is the $p$-th coefficient of the Chern character and $i_p$ is the map induced
by the identity map of $K_0(X)_{num}$. Note that it follows that
$\ch_p:\gr^p_FK_0(X)_{num}\otimes\QQ \to A^p(X)_{num}\otimes\QQ$ is an isomorphism.
Let us show that the inverse map is defined over $\ZZ$.

\begin{lemma}
There exists a linear map $\SO_p:A^p(X)_{num} \to \gr_F^p K_0(X)_{num}$,
which is inverse to $\ch_p$. Moreover, $\SO_p(Z) = [\CO_Z] \bmod F^{p+1}K_0(X)_{num}$.
\end{lemma}
\begin{proof}
For each $p$-cycle $Z$ on $X$ define $\SO_p(Z)$ as the image of the class of its structure sheaf in $\gr_F^p K_0(X)_{num}$.
Since
$$
\ch(\CO_Z) = Z + \text{terms of degree larger than $p$},\eqno{(\dagger)}
$$
we have $[\CO_Z] \in F^pK_0(X)_{num}$ and $\ch_p(\SO_p(Z)) = Z$.
Since $\ch_p$ is injective, it follows that $\SO_p$ is correctly defined
and $\ch_p\circ\SO_p = \id$.
\end{proof}

\begin{remark}\label{grpoz}
As we see from this Lemma the class of $[\CO_Z]$ in $\gr^p_FK_0(X)_{num}$ only depends
on the numerical class of $Z$.
\end{remark}

Further, it is easy to see that for any coherent sheaf $G$ supported in codimension $p$
the $p$-th coefficient of the Chern character is integer, $\ch_p(G) \in A^p(X)_{num}$.
Indeed, if $Z_1,\dots,Z_m$ are codimension $p$ components of $\supp G$ and $\ell_i$ is the length
of $G$ at generic point of $Z_i$ then $\ch_p(G) = \sum\ell_iZ_i$. Moreover, the same argument
shows that the map $\ch_p:\gr^p_SK_0(X)_{num} \to A^p(X)_{num}$ is surjective.
It follows that we have the following commutative diagram
$$
\xymatrix{
A^p(X)_{num} \ar@{^{(}->}[rr] \ar@{^{(}->}[drr]^{\SO_p} && A^p(X)_{num}\otimes\QQ \\
\gr^p_SK_0(X)_{num} \ar[rr]^{i_p} \ar@{->>}[u]^{\ch_p} && \gr^p_FK_0(X)_{num} \ar@{^{(}->}[u]^{\ch_p}
}
$$

\begin{proposition}\label{equiv}
The following properties for a smooth projective variety $X$ are equivalent:

\rnc{\theenumi}{\roman{enumi}}
\begin{enumerate}
\item
$\forall\ 0 \le p \le n$ $\SO_p$ is an isomorphism;
\item
$X$ is AK-compatible;
\item
$\forall\ 0 \le p \le n$ $S^pK_0(X)_{num} = F^pK_0(X)_{num}$;
\item
$\forall\ 0 \le p \le n$ $\ch_p(\gr^p_FK_0(X)_{num}) \subset A^p(X)_{num}$.
\end{enumerate}
\end{proposition}
\begin{proof}
(i) $\Rightarrow$ (ii):
An evident induction argument shows that $\{[\CO_{Z_q^i}]\}_{q \ge p}^{1\le i\le m_q}$ is a basis in $F^pK_0(X)_{num}$.

(ii) $\Rightarrow$ (iii):
Assume that $X$ is AK-compatible. Choose bases in all $A^p(X)_{num}$ as in Definition~\ref{AK}
and assume that $v = \sum a_q^i[\CO_{Z_q^i}] \in F^pK_0(X)_{num}$ for some $a_q^i \in \ZZ$.
Let $q$ be the minimal integer such that $a_q^i \ne 0$ for some $i$ and assume that $q < p$.
Then applying $\ch_q$ we see that
$$
0 = \ch_q(v) =
\sum_i a_q^i \ch_q([\CO_{Z_q^i}]) =
\sum_i a_q^i \ch_q(\SO_q(Z_q^i)) =
\sum_i a_q^i Z_q^i
$$
which implies $a_q^i = 0$ for all $i$. So, it follows that $q \ge p$, hence
$a_q^i = 0$ for all $q < p$. But then it is clear that $v \in S^pK_0(X)_{num}$.
So we see that $F^pK_0(X)_{num} \subset S^pK_0(X)_{num}$.


(iii) $\Rightarrow$ (iv):
If $S^pK_0(X)_{num} = F^pK_0(X)_{num}$ then $\gr^p_SK_0(X)_{num} = \gr^p_FK_0(X)_{num}$ for all $p$,
hence $\ch_p(\gr^p_FK_0(X)_{num}) = \ch_p(\gr^p_SK_0(X)_{num}) \subset A^p(X)_{num}$.

(iv) $\Rightarrow$ (i):
Since $\ch_p$ and $\SO_p$ are mutually inverse isomorphisms of $\gr^p_FK_0(X)_{num}\otimes\QQ$ and
$A^p(X)_{num}\otimes\QQ$ and preserve lattices $\gr^p_FK_0(X)_{num}$ and $A^p(X)_{num}$,
it follows that they induce isomorphisms of these lattices.
\end{proof}


%


Our next goal is to give several sufficient conditions for AK-compatibility.

\begin{lemma}\label{ch012n}
For any smooth projective variety $X$ we have $\ch_p(\gr^p_FK_0(X)_{num}) \subset A^p(X)_{num}$ for $p=0,1,2$ and $p=n$.
\end{lemma}
\begin{proof}
Note that $\ch_0(G)$ is the rank of $G$ and $\ch_1(G) = c_1(G)$, which implies the claim for $p=0$ and $p=1$.
For $p=2$ we have $\ch_2(G) = c_1(G)^2/2 - c_2(G)$, so if $c_1(G) = \ch_1(G) = 0$, then $\ch_2(G) = -c_2(G) \in A^2(X)_{num}$.
Finally, if $G \in F^nK_0(X)_{num}$ then by Riemann--Roch we have $\ch_n(X) = \chi(\CO_X,G) \in \ZZ$,
where we have identified $A^n(X)_{num}$ with $\ZZ$ via the degree map. This proves the claim for $p=n$.
\end{proof}

\begin{corollary}\label{ak3}
If $\dim X \le 3$ then $X$ is AK-compatible.
\end{corollary}

Another approach to AK-compatibility is given by the following

\begin{lemma}\label{k0}
Assume that the intersection pairing $A^p(X)_{num} \otimes A^{n-p}(X)_{num} \to \ZZ$
induces an isomorphism $A^p(X)_{num} \to A^{n-p}(X)_{num}^*$. Then the map
$\SO_p:A^p(X)_{num} \to \gr_F^p K_0(X)_{num}$ is an isomorphism.
In particular, if the intersection pairing $A^p(X)_{num} \otimes A^{n-p}(X)_{num} \to \ZZ$
induces an isomorphism $A^p(X)_{num} \to A^{n-p}(X)_{num}^*$ for all $p$ then $X$ is AK-compatible.
%
\end{lemma}
\begin{proof}
Let $Z,W \subset X$ be subschemes of codimension $p$ and $(n-p)$ respectively.
Then~$(\dagger)$ and Riemann--Roch implie that
$$
Z\cdot W = (-1)^p\chi(\CO_Z,\CO_W),
$$
hence we have a commutative diagram
$$
\xymatrix{
A^p(X)_{num} \ar[rr]^{\SO_p} \ar[d]_\cdot && \gr_F^p K_0(X)_{num} \ar[d]^{(-1)^p\chi} \\
A^{n-p}(X)_{num}^* && \gr_F^{n-p} K_0(X)_{num}^* \ar[ll]_{\SO_{n-p}^*}
}
$$
Note that all the maps are finite index embeddings. So, if the left vertical arrow
is an isomorphism then $\SO_p$ is also an isomorphism.
\end{proof}

The results of this section allow to describe $K_0(X)_{num}$
for all Fano threefolds.

\begin{corollary}\label{k0f3}
Let $V$ be a Fano threefold with $\Pic V \cong \ZZ$.
Let $H$ be the generator of $\Pic V$, $L$ a line on $V$, and $P$ a point on $V$.
Then $K_0(X)_{num} = \langle [\CO_V], [\CO_H], [\CO_L], [\CO_P] \rangle$.
\end{corollary}
\begin{proof}
We can argue either by Corollary~\ref{ak3} or by Lemma~\ref{k0} that $V$ is AK-compatible.
Hence by definition of AK-compatibility we obtain the required basis.
%
\end{proof}

\begin{remark}
One can combine the results of Lemma~\ref{ch012n} and Lemma~\ref{k0} for the verification
of AK-compatibility. In other words, if for an algebraic variety $X$ the conditions of Lemma~\ref{k0}
are true for all $3 \le p \le n-1$ then $X$ is AK-compatible. Indeed, it is easy to see from the proof
of Proposition~\ref{equiv} that the property (iv${}_p$) for each $p$ implies the property (i${}_p$).

These considerations apply e.g. for cubic fourfolds. Indeed, for $p = 3$ the conditions of Lemma~\ref{k0}
are true, hence the cubic fourfold is AK-compatible.
\end{remark}

\end{document}